\documentclass[a4paper]{amsart}

\usepackage{mathtools} 
\usepackage{amssymb,amsmath,amsthm}
\usepackage{graphicx,color}
\usepackage{caption}
\usepackage{subfigure} 
\usepackage{subfigure} 
\usepackage{float}
\usepackage{a4wide}
\usepackage[colorlinks=false]{hyperref}

\theoremstyle{plain}
\newtheorem{theorem}{Theorem}[section]
\newtheorem{lemma}[theorem]{Lemma}
\newtheorem{proposition}[theorem]{Proposition}

\theoremstyle{definition}

\theoremstyle{remark}
\newtheorem{remark}[theorem]{Remark}

\numberwithin{equation}{section}
\newcommand{\N}{\mathbb N} 
 
\newcommand{\R}{\mathbb R}

\newcommand{\diam}{\operatorname{diam}}

\newcommand{\E}{{\mathcal E}}

\newcommand{\F}{{\mathcal F}}

\newcommand{\Ra} {\Rightarrow}

\newcommand{\eps}{\varepsilon}

\newcommand{\dx}{\,\mathrm{d}x}

\newcommand{\dy}{\,\mathrm{d}y}
\newcommand{\ds}{\,\mathrm{d}s}

\def\XXint#1#2#3{{\setbox0=\hbox{$#1{#2#3}{\int}$ }
\vcenter{\hbox{$#2#3$ }}\kern-.6\wd0}}

\makeatletter
\@namedef{subjclassname@2020}{2020 Mathematics Subject Classification}
\makeatother

\date{\today}
\title[Connected Coulomb Columns]{Connected Coulomb Columns: Analysis and Numerics}

\author{P.\ Dondl}
\address{Patrick Dondl\\Abteilung f\"ur Angewandte Mathematik\\Albert-Ludwigs-Universit\"at Freiburg\\Hermann-Herder-Str.~10\\ 79104 Freiburg i.~Br.\\Germany}
\email{patrick.dondl@mathematik.uni-freiburg.de}

\author{M.\ Novaga}
\address{Matteo Novaga\\ Dipartimento di Matematica Universit\`a di Pisa\\ Largo Bruno Pontecorvo 5\\ 56127 Pisa\\ Italy}
\email{matteo.novaga@unipi.it}

\author{S.\ Wojtowytsch}
\address{Stephan Wojtowytsch\\
Program in Applied and Computational Mathematics\\
Princeton University\\
205 Fine Hall - Washington Road\\
Princeton, NJ 08544
}
\email{stephanw@princeton.edu}

\author{S.\ Wolff-Vorbeck}
\address{Steve Wolff-Vorbeck\\Abteilung f\"ur Angewandte Mathematik\\Albert-Ludwigs-Universit\"at Freiburg\\Hermann-Herder-Str.~10\\ 79104 Freiburg i.~Br.\\Germany}
\email{steve.wolff-vorbeck@mathematik.uni-freiburg.de}



\begin{document}

\begin{abstract}
We consider a version of Gamow's liquid drop model with a short range attractive perimeter-penalizing potential and a long-range Coulomb interaction of a uniformly charged mass in $\R^3$. Here we constrain ourselves to minimizing among the class of shapes that are columnar, i.e., constant in one spatial direction. Using the standard perimeter in the energy would lead to non-existence for any prescribed cross-sectional area due to the infinite mass in the constant spatial direction. In order to heal this defect we use a connected perimeter instead. We prove existence of minimizers for this connected isoperimetric problem with long-range interaction and study the shapes of minimizers in the small and large cross section regimes. For an intermediate regime we use an Ohta-Kawasaki phase field model with connectedness constraint to study the shapes of minimizers numerically.
\end{abstract}

\maketitle

\section{Introduction}

In this article we study an isoperimetric problem with an added long-range repulsive term in two space dimensions. The repulsive term can be seen as the interaction energy of a uniformly charged mass, when restricted to columnar states, i.e., for a given set $E\subset \R^2$, the charged mass is given as $\{(x_1,x_2,x_3)\subset \R^3 : (x_1,x_2)\in E, x_3\in\R\}$. This leads to a logarithmic interaction kernel. 

The study of charged droplets is commonly referred to as Gamow's liquid drop model \cite{gamow1928quantentheorie}, originally devised for an explanation for the shape of nuclear cores due to a competition between short range attractive (e.g., perimeter type) and long range repulsive (e.g., Coulombic) potentials. In many cases, this leads to existence of minimizers up to a critical mass (which then are usually spherical), and nonexistence thereafter (since, when breakup into more than one piece is energetically expedient, these pieces can further reduce the interaction energy by increasing their distance) \cite{MR3509012,MR3055587,MR3272365,MR3226747,MR3218265,MR3794526}. Nonexistence for larger masses can of course be prevented by considering droplets confined to bounded domains \cite{MR3069960} or background potentials, where breakup or loss of existence may again depend on the relative strengths of the confining and repulsive potentials \cite{MR3947064, MR3251907}.

In our setting, due to the restriction to columnar sets and a full Coulombic interaction, we are faced with nonexistence of minimizers for any prescribed area of a two-dimensional slice $E$ when considering a standard perimeter since the extension to three dimensions always has infinite mass. Instead of introducing a background potential, however, we opt for a different avenue, also pursued in \cite{dayrens2019connected}: we replace the standard perimeter with a `connected perimeter', which can be briefly described by the relaxation of the perimeter of a connected, $L^1$-approximating set. For a precise statement on the setting, see Section \ref{section preliminaries}. To conform with the common language in the mathematics literature concerning such charged mass models, we will, in the following, refer to the prescribed area of a two-dimensional slice as the `mass' in the problem -- this is a slight abuse of nomenclature, as it is in fact a mass density when considering the associated three-dimensional problem. Furthermore, even though we are talking about a minimization problem which is effectively two-dimensional, we will refer to the sets $E$ as charged `droplets'.

This connected perimeter used here was introduced in \cite{dayrens2019connected}. A phase-field variant of the connected perimeter was developed in \cite{MR4011685}.

Our main analytic results are the following. 

\begin{enumerate}
\item Minimizers exist for any mass. This part, in Section \ref{sec:existence}, follows closely the arguments in \cite{dayrens2019connected}.

\item A charged droplet of small mass aggregates in a disk. We call this the {\em perimeter-dominated} regime since minimization of surface tension drives the behavior. Our connectedness constraint does not influence the local behavior in this regime -- it does, however, remove the global minimizer of droplets disappearing at infinity in opposite directions. The analysis is conducted in Section \ref{section small mass} and the precise result is stated in Theorem \ref{prop2}. The main difficulty here is that the connected perimeter does not directly yield sufficient regularity for minimizers to study the Euler-Lagrange-equation of the energy.

\item Charged drops of large mass organize in long and thin objects. We call this the {\em repulsion-dominated} regime. This regime is considered in Section \ref{section large mass}. Precise statements are given in Theorems \ref{theorem long} and \ref{theorem thin}, together with an asymptotic expansion of the minimal energy in terms of problem parameters in Theorem \ref{theorem scaling}.
\end{enumerate}

To study the intermediate regime, we use an efficient numerical method to find shapes of minimizing configurations in numerical simulations using an Ohta-Kawasaki phase-field approximation of our problem, including the connectedness constraint. The phase-field simulations of course take place on bounded domains, so some aspects of confinement as well as further changes when requiring simple-connectedness are studied there as well.

The article is structured as follows. A rigoros introduction to the problem is given in Section \ref{section preliminaries}. We present analytic results concerning existence and shapes of minimizers of our functional in Section \ref{section analysis}, before developing the phase-field approach and discussing numerical results in Section \ref{numerics}.

\section{Preliminaries}\label{section preliminaries}
We study the variational problems associated with the energy functionals
\begin{align}\label{thefunc1}
\F^{\lambda}_{C}(\chi_{E}) \coloneqq \ \overline{P^{r}_{C}}(\{\chi_{E} =1\}) + \lambda \int\limits_{\mathbb{R}^{2}}\int\limits_{\mathbb{R}^{2}} \log\left(\frac{1}{|x-y|}\right)\chi_{E}(x)\chi_{E}(y) ~\mathrm{d}x\mathrm{d}y,
\end{align}
\begin{align}\label{thefunc2}
\F^{\lambda}_{S}(\chi_{E}) \coloneqq \overline{P^{r}_{S}}(\{\chi_{E} =1\}) + \lambda\int\limits_{\mathbb{R}^{2}}\int\limits_{\mathbb{R}^{2}} \log\left(\frac{1}{|x-y|}\right)\chi_{E}(x)\chi_{E}(y) ~\mathrm{d}x\mathrm{d}y,
\end{align}
where $\chi_{E}$ is the characteristic function of the subset $E\subset \mathbb{R}^{2}$ with finite perimeter and volume/mass $|E|= m>0$ and $\lambda > 0$ is a parameter. Here $\overline{P^{r}_{C}}$ and $\overline{P^{r}_{S}}$ describe the connected and the simply connected perimeter of the set $E$ defined by 
\begin{align*}
\overline{P^{r}_{C}}(E)=\inf \left\{ \liminf_{n\rightarrow \infty} P(E_{n}) \  | \  E_{n}\rightarrow E \ \text{in} \ L^{1}, E_{n} \ \text{connected and} \ C^{\infty}-\text{smooth}\right\},
\end{align*}
and 
\begin{align*}
\overline{P^{r}_{S}}(E)=\inf \left\{ \liminf_{n\rightarrow \infty} P(E_{n}) \ | \  E_{n}\rightarrow E \ \text{in} \ L^{1}, E_{n} \ \text{simply connected and} \ C^{\infty}-\text{smooth}\right\},
\end{align*}
where $P(\cdot)$ is the usual perimeter of a set.
It was shown recently \cite{dayrens2019connected}, that for essentially bounded sets $E\subset \mathbb{R}^{2}$ such that $\partial E =\partial_{*}E$ modulo sets of zero $\mathcal{H}^{1}$-measure the identity
\[
\overline{P^{r}_{C}}(E)=P(E)+2\mathcal{S}t(E)
\]
 holds, where $\mathcal{S}t(E)$ is the length of the Steiner tree of $\bar{E}$, i.e., 
\[
\mathcal{S}t(E)=\inf \{ \mathcal{H}^{1}(K)| E\cup K \ \text{connected} \}.
\]
 Above, $\mathcal{H}^{1}$ denotes the $1$-dimensional Hausdorff measure on $\mathbb{R}^{2}$ and $\partial_{*}E$ is the essential boundary of $\bar{E}$, see Definition 3.60 in~\cite{MR1857292}. For the existence of Steiner trees, their properties and regularity see~\cite{MR3018174}.

In the following, we consider the minimization problem
\begin{align}\label{min}
\min_{\substack{E\subset \mathbb{R}^{2}, \\ |E|=m}} \F^{\lambda}_{C/S}(\chi_{E}).
\end{align}
We prove that minimizers exist for all $ m > 0$ and all $\lambda > 0$ and describe the shape of such minimizers for small masses $m$/small parameters $\lambda$ and large masses $m$/large parameters $\lambda$, respectively.

\begin{remark}
We note that due to the unboundedness of the logarithmic potential, the functional $ P(\{\chi_{E} =1\}) + \lambda \int\limits_{\mathbb{R}^{2}}\int\limits_{\mathbb{R}^{2}} \log\left(\frac{1}{|x-y|}\right)\chi_{E}(x)\chi_{E}(y) ~\mathrm{d}x\mathrm{d}y$ on $\R^2$ does not admit minimizers for any $\lambda>0$ or any prescribed mass. Usage of the connected perimeter is therefore necessary for the arguments below.  
\end{remark}

\section{Analytical Results}\label{section analysis}

\subsection{Existence of Minimizers} \label{sec:existence}
We prove the existence of solutions of the problem in~\eqref{min} for all masses $m>0$ and all $\lambda > 0$. This section follows the proof in \cite{dayrens2019connected}, where the nonlocal part of the energy was given by
\begin{align*}
\int\limits_{\mathbb{R}^{2}}\int \limits_{\mathbb{R}^{2}} \frac{1}{|x-y|^{\alpha}}\chi_{E}(x)\chi_{E}(y) ~\mathrm{d}x\mathrm{d}y
\end{align*}
with some $\alpha \in (0,2)$. In \cite{dayrens2019connected}, the singularity at the origin is stronger, leading to a stronger local repulsion. On the other hand, the power law interaction decays to zero at infinity while the logarithm of inverse distance approaches negative infinity. This leads to a stronger non-local ``attraction from infinity'' in our model. While both interaction models drive sets to be more `spread out', the precise mechanisms are different.

To adapt the proof from Theorem 5.2 in~\cite{dayrens2019connected} to our case we first state the following simple proposition.
\begin{proposition}\label{prop0}
 For $m>0$, $R>0$ we have
\begin{align*}
\inf_{\substack {E \subset B_{R}(0), \\ |E|=m}}\lambda \int\limits_{\mathbb{R}^{2}} \int\limits_{\mathbb{R}^{2}} \log\left(\frac{1}{|x-y|}\right)\chi_{E}(x)\chi_{E}(y) ~\mathrm{d}x\mathrm{d}y \geq - m^{2}\lambda \log(\diam E) \geq - m^{2}\lambda \log(2R). 
\end{align*}
\end{proposition}
\begin{proof}
The inequality follows from basic estimates on the logarithm and the diameter.
\end{proof}
\begin{remark}
Here and below we denote the diameter of a measurable set $E$ as
\[
\diam(E) = \sup\left\{|x-y|\::\: |E\cap B_r(x)|, \:|E\cap B_r(y)| >0\quad\forall\ r>0\right\}.
\]
As usual, $B_R(x)$ denotes the ball of radius $R$ around a point $x$, in our case always in $\R^2$.
\end{remark}

We also require a continuity result, analogous to \cite[Lemma 5.1]{dayrens2019connected}.
\begin{lemma}\label{lem3}
Let $R>0$ and let $E_n$ be a sequence of sets such that
\[
E_n\subseteq B_R(0), \qquad \chi_{E_n}\to \chi_E \quad\text{in }L^1(B_R(0)).
\]
Then 
\[
\int\limits_{E_n}\int\limits_{E_n} \log\left(\frac{1}{|x-y|}\right)  ~\mathrm{d}x\mathrm{d}y \to \int\limits_{E}\int\limits_{E} \log\left(\frac{1}{|x-y|}\right)  ~\mathrm{d}x\mathrm{d}y.
\]
\end{lemma}
\begin{proof}
The logarithm of inverse distance is integrable on $B_R(0)\times B_R(0)$, so the result follows from the dominated convergence theorem.
\end{proof}

This allows to prove the statement ensuring existence of minimizers for any mass.
\begin{theorem}\label{theorem existence of minimizers}
For all $m > 0$ and all $\lambda > 0$ the minimization problems 
\begin{align}\label{min1}
\min \left \{ \overline{P^{r}_{C}}(E) + \lambda \int\limits_{E}\int\limits_{E} \log\left(\frac{1}{|x-y|}\right) ~\mathrm{d}x\mathrm{d}y \ \bigg| \ E\subset \mathbb{R}^{2} \ \text{measurable},|E|=m\right\},
\end{align}
\begin{align}\label{min2}
\min \left \{ \overline{P^{r}_{S}}(E) + \lambda \int\limits_{E}\int\limits_{E} \log\left(\frac{1}{|x-y|}\right) ~\mathrm{d}x\mathrm{d}y \ \bigg| \ E\subset \mathbb{R}^{2} \ \text{measurable},|E|=m\right\}
\end{align}
for $\F_C^\lambda$ and $\F_S^\lambda$, respectively, admit solutions. Furthermore, there exists a constant $C>0$, depending only on $m$ and $\lambda$, so that any solution $E$ to the minimization problem above satisfies $\diam(E) < C$.
\end{theorem}

\begin{proof}
We consider the case of $\F_C^\lambda$, the proof for $\F_S^\lambda$ proceeds analogously. 

Before proceeding we mention that for any connected set $E$ of finite perimeter, we have 
\begin{align}\label{inequality}
2\operatorname{diam}(E) \leq |\partial E|
\end{align}
and thus obtain, using Proposition \ref{prop0},
\begin{align}\label{eq1}
\mathcal{F}^{\lambda}_{C}(\chi_E) \geq  2\operatorname{diam}(E) - \lambda m^{2} \log(\operatorname{diam}(E)) \geq \min_{r>0} (2r-\lambda m^{2} \log(r)) > -\infty.
\end{align}

Let now $E_{n}$ be a minimizing sequence for the problem in~\eqref{min1} with $|E_{n}|=m < \infty$. We immediately obtain from \eqref{eq1} that $\sup_n\diam(E_n) < C(m,\lambda)$. Using the translation invariance of $\F^{\lambda}_{C}$, we may thus assume that there exists $R>0$ such that $E_n\subset B_R(0)$ for all $n\in\N$, modulo Lebesgue null sets. The usual compactness of sets of finite perimeter, together with Lemma \ref{lem3}, yield the existence result.
\end{proof}

\begin{remark} \label{rem:scaling_bound}
Note that by approximation, estimate \eqref{eq1} carries over to the minimizer, denoted by $E_{m,\lambda}$, so that
\begin{equation}\label{eq scaling lower bound}
\F_C(\chi_{E_{m,\lambda}}) \geq  2\operatorname{diam}(E_{m,\lambda}) - \lambda m^{2} \log(\operatorname{diam}(E_{m,\lambda})).
\end{equation}
\end{remark}

\subsection{Shape of Minimizers for Small Mass}\label{section small mass}
Now we consider the shape of solutions of~\eqref{min1} and~\eqref{min2} for small masses $m$ or small values of $\lambda > 0$ respectively. The goal of this section is to show that for small values $\lambda >0$ or small masses $m > 0$ the unique solution of~\eqref{min1} and~\eqref{min2} is a disk. 

We first note that, for a given set $E$ and $\mu>0$, we have
\[
|\mu E| = \mu^2\,|E|, \qquad \overline{P^r_C}(\mu E) = \mu\,\overline{P^r_C}(E)
\]
since $\overline{P^r_C}$ scales like the perimeter functional, and
\begin{align*}
\int_{\mu E} \int_{\mu E} \log\left(\frac1{|x-y|}\right) \dx \dy 
	&= \int_{E} \int_{E} \log\left(\frac1{|\mu x-\mu y|}\right) \mu^2\dx \,\mu^2\dy\\
	&= \mu^4 \int_E\int_E \log\left(\frac1{|x- y|}\right) + \log \left(\frac1{\mu}\right)\dx\dy\\
	&= \mu^4 \int_E\int_E \log\left(\frac1{|x- y|}\right)\dx\dy - \mu^4\log(\mu) \,|E|^2.
\end{align*}
The last term on the right hand side is independent of $E$ when $|E|=m$ is fixed. Thus the minimization problems 
\[
\text{minimize}\quad \overline{P^r_C}(E) + \lambda \int_{E} \int_{E} \log\left(\frac1{|x-y|}\right) \dx \dy
\]
in the class of sets with mass $|E|=m$ and 
\[
\text{minimize}\quad \sqrt{\frac m\pi} \cdot \left[\overline{P^r_C}(E) + \lambda\left(\frac m\pi\right)^{3/2} \int_{E} \int_{E} \log\left(\frac1{|x-y|}\right) \dx \dy\right]
\]
in the class of sets with mass $|E|= \pi$ are equivalent by rescaling with $\mu = \sqrt{\frac m\pi}$. We conclude that considering the small mass limit for fixed $\lambda$ and the small $\lambda$ limit for fixed mass are equivalent and thus confine ourselves to the case where the mass $|E|=m=\pi$ is fixed and $\lambda$  becomes small.
\begin{remark}
Of course, the scaling argument remains unchanged if $\overline{P^r_C}$ is replaced by the standard perimeter $P$.
\end{remark}
We now consider the functional 
\begin{align}\label{thefunc_const}
\F^{\lambda}(\chi_{E})\coloneqq P(E) + \lambda \int\limits_{E}\int\limits_{E} \log\left(\frac{1}{|x-y|}\right) ~\mathrm{d}x\mathrm{d}y 
\end{align}
and the constrained optimization problem
\begin{align}\label{min_const}
\min \left \{ \F^{\lambda}(\chi_{E}) \bigg| \ E\subset \mathbb{R}^{2} \ \text{measurable},|E|=\pi, E\subset B_\rho \right \},
\end{align}
with $\rho>2C(\pi,\lambda)$, the diameter bound from Theorem \ref{theorem existence of minimizers}. Existence of solutions for this problem immediately follows from standard existence theory of minimizers for lower semicontinuous functionals bounded from below. Let now $E$ be a minimizer of $\F^{\lambda}(\chi_{E})$ in the class above and  $E'\subset B_\rho(0)$. We then obtain
\begin{align}\label{eq:quasi_min1}
\int\limits_{E}\int\limits_{E} \log\left(\frac{1}{|x-y|}\right)~\mathrm{d}x\mathrm{d}y &- \int\limits_{E'}\int\limits_{E'}\log\left(\frac{1}{|x-y|}\right)~\mathrm{d}x\mathrm{d}y \notag \\ 
&= \int\limits_{\mathbb{R}^{2}}\int\limits_{\mathbb{R}^{2}} \log\left(\frac{1}{|x-y|}\right) \chi_{E}(x)\chi_{E}(y)~\mathrm{d}x\mathrm{d}y \notag \\
 &- \int\limits_{\mathbb{R}^{2}}\int\limits_{\mathbb{R}^{2}} \log\left(\frac{1}{|x-y|}\right) \chi_{E'}(x)\chi_{E'}(y)~\mathrm{d}x\mathrm{d}y \notag \\
&=  \int\limits_{\mathbb{R}^{2}}\int\limits_{\mathbb{R}^{2}} \log\left(\frac{1}{|x-y|}\right) \chi_{E}(x)(\chi_{E}(y)-\chi_{E'}(y))~\mathrm{d}x\mathrm{d}y \notag \\
&+\int\limits_{\mathbb{R}^{2}}\int\limits_{\mathbb{R}^{2}} \log\left(\frac{1}{|x-y|}\right) \chi_{E'}(x)(\chi_{E}(y)-\chi_{E'}(y))~\mathrm{d}x\mathrm{d}y. 
\end{align}
Taking now $\omega_{1} = \log\left(\frac{1}{|x-y|}\right) * \chi_{E}$ and $\omega_{2}= \log\left(\frac{1}{|x-y|}\right) * \chi_{E'}$ we can estimate 
\begin{align}\label{eq:quasi_min2}
\int\limits_{E}\int\limits_{E} \log\left(\frac{1}{|x-y|}\right)~\mathrm{d}x\mathrm{d}y&- \int\limits_{E'}\int\limits_{E'}\log\left(\frac{1}{|x-y|}\right)~\mathrm{d}x\mathrm{d}y \notag \\
&\leq \int\limits_{\mathbb{R}^{2}}\omega_{1}(y)(\chi_{E}(y)-\chi_{E'}(y)) ~\mathrm{d}y + \int\limits_{\mathbb{R}^{2}} \omega_{2}(y)(\chi_{E}(y)-\chi_{E'}(y)) ~\mathrm{d}y \notag \\
&\leq \|\omega_{1}\|_{L^{\infty}(E\Delta E')}|E\Delta E'|+\|\omega_{2}\|_{L^{\infty}(E\Delta E')}|E\Delta E'|.
\end{align}
Finally using that $\|\omega_{1}\|_{L^{\infty}(E\Delta E')} \leq \|\omega_{1}\|_{L^{\infty}(B_\rho(0))}$ and $\|\omega_{2}\|_{L^{\infty}(E\Delta E')} \leq \|\omega_{2}\|_{L^{\infty}(B_\rho(0))}$ we have 
\begin{align*}
P(E)\leq P(E') + \left(|\omega_{1}\|_{L^{\infty}(B_\rho(0))}+\|\omega_{2}\|_{L^{\infty}(B_\rho(0))}\right) |E\Delta E'| \leq P(E') + C |E\Delta E'|. 
\end{align*}
Thus minimizers of $\F^{\lambda}(\chi_{E})$ are quasi-minimizers of the perimeter constrained to lie within a smooth set. It follows that for a solution $E$ of~\eqref{min_const} the boundary $\partial E$ is of class $C^{1,1}$ \cite{MR664576}, see also \cite{MR2976521}.
Thus, recalling the bound on the diameter, 
we conclude that every minimizer has a well-defined curvature, and the boundary satisfies the Euler-Lagrange equation for~\eqref{thefunc_const}. Therefore, the procedure from~\cite{MR3055587} is applicable, and we obtain the following result. 
\begin{proposition}\label{prop1}
There exists $\lambda_1$ such that for all $\lambda \leq \lambda_1$ the unique solution of~\eqref{min_const} is the unit disk. 
\end{proposition} 
\begin{proof}
Let $B_{1}=B_{1}(x_{0})$, where $x_{0}$ is the barycenter of $E$. Due to the regularity of $\partial E$ and the uniform boundedness constraint, $E\subset B_{\rho}(0)$, we can apply \cite[Lemma 7.2]{MR3055587}, \cite[Lemma 7.3]{MR3055587} and \cite[Lemma 7.4]{MR3055587} to the problem in \eqref{min_const}. Thus we deduce, that for small $\lambda$ the set $E$ is convex and fulfills 
\begin{align}\label{eq5.0}
|E\Delta B_{1}(x_{0})|\leq C \sqrt{D(E)}
\end{align}
with a universal constant $C> 0$ and the isoperimetric deficit $D(E)$ given by 
\[
D(E)\coloneqq \frac{|\partial E|}{2\pi}-1. 
\] 
Now we have $\F^{\lambda}(\chi_{E}) \leq \F^{\lambda}(\chi_{B_{1}})$ by the minimization property of $E$, which is equivalent to  
\begin{align}
D(E) \leq \frac{\lambda}{2\pi}\left( \;\int\limits_{B_{1}}\int\limits_{B_{1}} \log\left(\frac{1}{|x-y|}\right)~\mathrm{d}x\mathrm{d}y- \int\limits_{E}\int\limits_{E}\log\left(\frac{1}{|x-y|}\right)~\mathrm{d}x\mathrm{d}y\right).
\end{align}
Transforming like in~\eqref{eq:quasi_min1} we get
\begin{align*}
\int\limits_{B_{1}}\int\limits_{B_{1}} \log\left(\frac{1}{|x-y|}\right)~\mathrm{d}x\mathrm{d}y&- \int\limits_{E}\int\limits_{E}\log\left(\frac{1}{|x-y|}\right)~\mathrm{d}x\mathrm{d}y \\
&\leq \int\limits_{\mathbb{R}^{2}}\omega_{1}(y)(\chi_{B_{1}}(y)-\chi_{E}(y)) ~\mathrm{d}y + \int\limits_{\mathbb{R}^{2}} \omega_{2}(y)(\chi_{B_{1}}(y)-\chi_{E}(y)) ~\mathrm{d}y \\
&\leq \|\omega_{1}\|_{L^{\infty}(B_{1}\Delta E)}|B_{1}\Delta E|+\|\omega_{2}\|_{L^{\infty}(B_{1}\Delta E)}|B_{1}\Delta E|
\end{align*}
with $\omega_{1} = \log\left(\frac{1}{|x-y|}\right) * \chi_{B_{1}}$ and $\omega_{2}= \log\left(\frac{1}{|x-y|}\right) * \chi_{E}$.
Using now that due to the convexity of $E$ and due to $\pi =|E|=|B_{1}|$ we have
\begin{align}\label{eq5.1}
\|\omega_{2}\|_{L^{\infty}(B_{1}\Delta E)}\leq C \|\omega_{1}\|_{L^{\infty}(B_{1}\Delta E)} \leq C |B_{1}\Delta E|
\end{align}
with a constant $C$ independent of $E$, see results from~\cite{MR2456887}, we can finally conclude 
\begin{align}\label{eq5.2}
D(E) \leq C \lambda |B_{1}\Delta E|^{2}.
\end{align}
Combining \eqref{eq5.0} and \eqref{eq5.2}, we obtain 
\begin{align}\label{eq5}
c |B_{1}\Delta E|^{2} \leq D(E) \leq C \lambda  |B_{1}\Delta E|^{2}
\end{align}
for some constants $c$ and $C$, which are independent of $E$. This means that as long as $\lambda$ is small enough we have $D(E)=0$ and thus find $E=B_{1}(x_{0})$. 
\end{proof}

We can now use the above result to reach a conclusion for the connected perimeter, but without constraint.
\begin{theorem}\label{prop2}
There exists $\lambda_{0}$ such that for all $\lambda \leq \lambda_{0}$ the unique minimizers of $\F_{C}^{\lambda}(\chi_{E})$ and $\F_{S}^{\lambda}(\chi_{E})$ are the unit disk. 
\end{theorem} 
\begin{proof}
We restrict ourselves to $\F_{C}^{\lambda}(\chi_{E})$ as the result for $\F_{S}^{\lambda}(\chi_{E})$ follows in the same manner and consider only $\lambda < 1$. Let again $B_{1}=B_{1}(x_{0})$ with the barycenter $x_{0}$ of $E$.
As derived in Theorem \ref{theorem existence of minimizers}, minimizers $E$ of $\F_{C}^{\lambda}(\chi_{E})$ satisfy $\diam(E)\le C$ with a uniform bound, so we may assume that up to translation we have $E\subset B_{2C}(0)$. Now that 
\[
P(E)\leq \overline{P^{r}_{C}}(E)
\]
for all $E\subset B_{2C}(0)$ we have 
\[
\F^{\lambda}(B_{1})\leq\F^{\lambda}(E)\leq \F_{C}^{\lambda}(E)
\]
for all $E\subset B_{2C}(0)$ by Proposition~\ref{prop1}. The assertion then follows from $P(B_{1})=\overline{P^{r}_{C}}(B_{1})$.
\end{proof}

\subsection{Shape of Minimizers for Large Mass}\label{section large mass}

In this section, we consider the large mass limit as the opposite extreme. We obtain the first two contributions to the energy expansion in the large length limit. The shape of minimizers is described only coarsely, showing that their diameter scales weakly like $\lambda$ and that they do not concentrate mass close to any point on scales significantly smaller than $O(\lambda)$. It remains open whether they are convex, and how minimizers look in the intermediate regime. Denote
\begin{align}
e_C(\lambda)&:= \min \left \{ \F_C^\lambda(\chi_E)\:\big|\: E\subseteq \R^2\text{ measurable, }|E| = \pi\right\},\\
e_S(\lambda)&:=  \min \left \{ \F_S^\lambda(\chi_E)\:\big|\: E\subseteq \R^2\text{ measurable, }|E| = \pi\right\}.
\end{align}

\begin{theorem}\label{theorem scaling}
For large $\lambda$, the expansions
\begin{align*}
e_C(\lambda) &= -\pi^2\,\lambda\,\log\left(\frac{\pi^2\lambda}2\right) + \tilde e_C(\lambda)\,\lambda + O(1)\\
e_S(\lambda) &= -\pi^2\,\lambda\,\log\left(\frac{\pi^2\lambda}2\right) + \tilde e_S(\lambda)\,\lambda + O(1)
\end{align*}
hold for coefficient functions satisfying 
\[
\pi^2 \:\leq\: \tilde e_C, \tilde e_S \:\leq\: 3\pi^2.
\]
\end{theorem}

\begin{proof}
{\bf Lower bound.} We deduce from \eqref{eq scaling lower bound} that
\[
e_C(\lambda) \geq \inf_{r>0} \big[2r - \pi^2\lambda\,\log(r)\big].
\]
The right hand side is large for $r$ close to zero or very large, so a minimizing $r$ exists and satisfies
\[
0 = \frac{d}{dr}\big[2r - \pi^2\lambda\,\log(r)\big] = 2 - \frac{\pi^2\lambda}r\qquad\Ra\quad r = \frac{\pi^2\lambda}2.
\]
It follows that
\[
e_C(\lambda) \geq 2\,\frac{\pi^2\lambda}2 - \pi^2\lambda\,\log\left(\frac{\pi^2\lambda}2\right) \geq \pi^2\lambda\left[1 - \log\left(\frac{\pi^2\lambda}2\right)\right].
\]

{\bf Upper bound.} Denote $E^r = [0,r]\times [0,\pi r^{-1}]$. We compute
\[
\overline{P_C^r}(E^r) = \overline{P_S^r}(E_r) = P(E^r) = 2r + 2\pi\,r^{-1}
\]
and
\begin{align*}
\int_{E^r}\int_{E^r} \log\left(\frac1{|x-y|}\right)\dx\dy &\leq \int_{E^r}\int_{E^r} \log\left(\frac1{|x_1-y_1|}\right)\dx\dy\\
	&= \frac{\pi^2}{r^2} \int_0^r\int_0^r \log\left(\frac1{|x_1-y_1|}\right)\dx\dy\\
	&\leq \frac{\pi^2}{r^2} \int_0^r \int_0^r \log\left(\frac1{|x_1-r/2|}\right)\dx_1\dy_1\\
	&= \frac{2\pi^2}r \int_0^{r/2}\log\left(\frac1s\right)\ds\\
	&= - \frac{2\pi^2}r\,\int_0^{r/2}\log (s)\ds\\
	&= - \frac{2\pi^2}r \left[\frac{r}2\log\left(\frac r2\right) - \frac r2 +0\right]\\
	&= \pi^2\left[1-\log\left(\frac r2\right)\right].
\end{align*}
Taking $r = \pi^2\lambda$, we obtain
\[
e_C(\lambda) \leq \E_C^\lambda(E^\lambda) \leq 2\pi^2\lambda+ \frac{2}{\pi\lambda} + \pi^2\lambda\left[1-\log\left(\frac {\pi^2\lambda}{2}\right)\right].
\]

{\bf Conclusion.} We have shown that
\[
\pi^2\lambda\left[1 - \log\left(\frac{\pi^2\lambda}2\right)\right]
	\: \leq\: e_C(\lambda) \: \leq \: 2\pi^2\lambda + \frac{2\pi}\lambda + \pi^2\lambda\left[1-\log\left(\frac {\pi^2\lambda}{2}\right)\right].
\]
The upper and lower bound differ by $2\pi^2\lambda + \frac{2\pi}\lambda$.
\end{proof}

The energy competitors we constructed were long, thin squares. To leading order, the only important property of the sequence was that most mass in the system has a distance of order $\lambda$ to the point $x$ for any $x\in E$. The precise distance is irrelevant since $\log(c\lambda) - \log\lambda = \log c \ll \log\lambda$ for large $\lambda$. The same scaling is expected for any sequence of sets with similar properties, for example annular regions with very similar (large) radii, or a union of several fattened line segments meeting at the origin. This zeroth order analysis therefore cannot provide more precise information on the shape of minimizers in the large mass regime.

We can, however, show that minimizers must be long and thin in a suitable sense. First, we show that the length of minimizers scales roughly like $\lambda$.

\begin{theorem}\label{theorem long}
\begin{enumerate}
\item Let $E^\lambda$ be a sequence of sets such that 
\[
\limsup_{\lambda\to \infty} \frac{\F_C^\lambda(E^\lambda)}{\lambda\,\log\lambda} \leq 0.
\]
Then 
\[
\limsup_{\lambda\to \infty} \frac{\log(\diam(E^\lambda))}{\log\lambda} \leq 1.
\]

\item If 
\[
\limsup_{\lambda\to \infty} \frac{\F_C^\lambda(E^\lambda)}{\lambda\,\log\lambda} = -\pi^2,
\]
then 
\[
\lim_{\lambda\to \infty} \frac{\log(\diam(E^\lambda))}{\log\lambda} =1.
\]
\end{enumerate}
\end{theorem}

\begin{proof}
Denote the diameter of $E^\lambda$ by $R^\lambda$. The intuition is as follows: If $R_\lambda\gg \lambda$, then the perimeter term becomes more expensive than the repulsion term can compensate for. On the other hand, if $R_\lambda\ll \lambda$, then we do not exploit the repulsion term fully.

{\bf Step 1.} In this step we show that if $\F_C^\lambda(E^\lambda) \leq 0$ for all large enough $\lambda$, then
\[
\limsup_{\lambda\to \infty} \frac{\log(R_\lambda)}{\log\lambda} \leq 1.
\]
We pass to a subsequence in $\lambda$ which realizes the upper limit.
Assume for the sake of contradiction that 
\[
1 < 1+2\sigma = \lim_{\lambda\to \infty} \frac{\log(R_\lambda)}{\log\lambda}  \leq C <\infty.
\]
Then
\[
R_\lambda = e^{\log(R_\lambda)} \geq e^{(1+\sigma)\,\log\lambda} = \lambda^{1+\sigma}
\]
for all sufficiently large $\lambda$, and thus by \eqref{eq scaling lower bound}
\[
\F^\lambda_C(E^\lambda) \geq 2R_\lambda - \pi^2\lambda\,\log(R_\lambda) \geq 2\,\lambda^{1+\sigma} - C\,\lambda\,\log\lambda>0
\]
for all sufficiently large $\lambda$. The assumption that the upper limit is finite can be removed by considering the splitting
\[
R_\lambda = R_\lambda^\frac1{1+\sigma/2} R_\lambda^\frac{\sigma}{2+\sigma} \geq \lambda^\frac{1+\sigma}{1+\sigma/2} R_\lambda^\frac{\sigma}{2+\sigma} \gg \lambda\,\log(R_\lambda).
\]

{\bf Step 2.} Assume now that
\[
\liminf_{\lambda\to\infty} \frac{\log(R_\lambda)}{\log\lambda} \leq 1-2\sigma <1.
\]
Again, we pass to a subsequence along the lower limit is realized. Since $E^\lambda$ is contained in a ball of radius $R_\lambda$, by Proposition \ref{prop0} we have
\[
 \int_{E^\lambda}\int_{E^\lambda}\log\left(\frac1{|x-y|}\right)\dx\dy \geq - \pi^2\log(R_\lambda) \geq - \pi^2\log\big(\lambda^{1-\sigma}\big) = -\pi^2(1-\sigma)\,\log\lambda
\]
for all sufficiently large $\lambda$, so
\[
\liminf_{\lambda\to\infty}\frac{\F^\lambda_C(E^\lambda)}{\lambda\,\log\lambda} \geq \liminf_{\lambda\to \infty} \frac1\lambda \int_{E^\lambda}\int_{E^\lambda}\log\left(\frac1{|x-y|}\right)\dx\dy \geq -\pi^2(1-\sigma) > -\pi^2.
\]

\end{proof}

In the next theorem, we show that minimizing sets are thin in the sense that mass does not concentrate close to a single point on a scale significantly shorter than $\lambda$.

\begin{theorem}\label{theorem thin}
For $\lambda>0$, let $E^\lambda$ be a family of sets such that for every $\lambda$ there exists a point $x^\lambda$ such that
\[
\big|E\cap B_{\lambda^\alpha}(x^\lambda)\big|\geq \bar c
\]
where $\alpha\in[0,1)$ and $\bar c>0$ does not depend on $\lambda$. Then
\[
\liminf_{\lambda\to\infty} \frac{ \F^\lambda_C(E^\lambda)}{\lambda\,\log\lambda} > -\pi^2.
\]
\end{theorem}

\begin{proof}
Up to translation, we may assume that $x^\lambda\equiv 0$. We pass to a subsequence in $\lambda$ such that
\[
\lim_{\lambda\to\infty} \big|E^\lambda \cap B_{\lambda^\alpha}\big| = \bar c>0,
\]
where we denote $B_{\lambda^\alpha} = B_{\lambda^\alpha}(0)$. 
We split
\begin{align*}
\int_{E^\lambda}\int_{E^\lambda}\log\left(\frac1{|x-y|}\right)\dx \dy &= \int_{E^\lambda \cap B_{\lambda^\alpha}}\int_{E^\lambda\cap B_{\lambda^\alpha}}\log\left(\frac1{|x-y|}\right)\dx \dy\\
	&\qquad + 2\int_{E^\lambda\cap B_{\lambda^\alpha}}\int_{E^\lambda\setminus B_{\lambda^\alpha}}\log\left(\frac1{|x-y|}\right)\dx \dy \\ 
	&\qquad+ \int_{E^\lambda\setminus B_{\lambda^\alpha}}\int_{E^\lambda\setminus B_{\lambda^\alpha}}\log\left(\frac1{|x-y|}\right)\dx \dy\\
	&\geq -\big|E^\lambda \cap B_{\lambda^\alpha}\big|^2\,\log\left(2\lambda^\alpha \right)\\
	&\qquad - 2\,|E^\lambda \cap B_{\lambda^\alpha}|\,|E^\lambda\setminus B_{\lambda^\alpha}|\,\log(R_\lambda)\\
	&\qquad- |E^\lambda\setminus B_{\lambda^\alpha}|^2 \log(R_\lambda)
\end{align*}
using Proposition \ref{prop0} on the first term. Assume for the sake of contradiction that
\[
\lim_{\lambda\to\infty} \frac{ \F^\lambda_C(E^\lambda)}{\lambda\,\log\lambda} = -\pi^2.
\]
Using Theorem \ref{theorem long}, we find that
\[
\lim_{\lambda\to\infty}\frac{\log (R_\lambda)}{\log\lambda} = 1,
\]
so
\begin{align*}
\liminf_{\lambda\to\infty}\frac1{\log\lambda} \int_{E^\lambda}\int_{E^\lambda}\log\left(\frac1{|x-y|}\right)\dx \dy
	&\geq - \bar c^2\,\alpha - 2(\pi - \bar c)\bar c - (\pi - \bar c)^2 > - \pi^2.
\end{align*}
We have thus reached a contradiction.
\end{proof}

So a minimizing sequence for $\F_C^\lambda$ has to be increasingly `spread out' and cannot concentrate positive mass close to a single point $x^\lambda$ on any scale $\lambda^\alpha$ for $\alpha<1$. 

Note that the arguments above are specific to the plane, and that the analysis changes entirely if the sets $E$ are confined to a bounded domain $\Omega$ or a compact manifold, e.g., a sphere or a flat torus. On such domains, the Green's function for the Laplacian is bounded from below and `spreading out' is no longer an option. Understanding minimizers analytically no longer seems possible in this regime.

In the following, we discuss a numerical approach to finding minimizers in the intermediate regime where $\lambda$ is neither small nor large, or where $E$ is confined to a set of finite diameter. In many cases, confinement to a domain is a feature of the problem.  To approximate minimization in the plane, we can minimize $E$ among sets confined to a a large set $\Omega$. If the confinement is to be neglected, the diameter of $\Omega$ has to scale linearly with $\lambda$.

\section{Numerical Implementation}\label{numerics}
\subsection{Variational Problem and Gradient Flow}
To describe the functionals from~\eqref{thefunc1} and~\eqref{thefunc2} in a diffuse interface approach suitable for numerical treatment, we consider the Ohta-Kawasaki free energy functional, first mentioned in~\cite{ohta1986equilibrium}, 
\begin{align}\label{functional}
\F_\eps^{\lambda}(u)=\frac{1}{c_{0}}\int\limits_{\Omega} \frac{\epsilon}{2}|\nabla u|^{2} + \frac{1}{\epsilon} W(u) ~\mathrm{d}x\mathrm{d}y  + \frac{\lambda}{2} \int\limits_{\Omega} (u-\bar{m})(-\Delta) ^{-1} (u-\bar{m}) ~\mathrm{d}x\mathrm{d}y,
\end{align}
where we set 
\begin{align*}
\bar{m}= \frac{1}{|\Omega|} \int\limits_{\Omega} u ~ \mathrm{d}x\mathrm{d}y, \quad W(s) = \frac{1}{4}s^{2}(s-1)^{2}, \quad\text{and}\quad c_{0}=\frac{1}{6\sqrt{2}}.
\end{align*}
As usual, the small parameter $\eps>0$ takes the role of diffuse interface width and $\Omega$ is a bounded domain. At least formally, $\F_\eps^{\lambda}$ is an approximation of $\F^\lambda$ from Section \ref{section small mass}, when one neglects the influence of the boundary on the electrostatic potential -- and therefore simply obtains a logarithmic potential. We note that by the results in \cite{MR3101793,MR3176350}, some of this correspondence can be made rigorous in appropriate scaling regimes. The study of the Ohta-Kawasaki functional, however, also has independent merit.

To give the inverse of the Laplacian a proper meaning we incorporate Neumann boundary conditions and define the operator $\Delta_{N}^{-1}$ such that the $H^{-1}(\Omega)$-inner product can be described in one of the equivalent forms 
\[
     \left<w,v\right>_{H^{-1}(\Omega)}\coloneqq \left\{\begin{array}{ll} (-\Delta_{N}^{-1}w,v)  \\
          \big((-\Delta_{N})^{-\frac{1}{2}}w,(-\Delta_{N})^{{-\frac{1}{2}}}v\big) \ \ \ \ \ \forall \ v,w \in H_{*}^{1}(\Omega), \\
 (w,-\Delta_{N}^{-1}v)          
          
          \end{array}\right.
  \]
where $-\Delta_{N}^{-1}v=g$ means that $-\Delta g=v$ with $g \in H_{*}^{1}(\Omega)$ and $\frac{\partial g}{\partial \eta}_{|\partial \Omega}=0$ with the outer unit normal $\eta$ on $\partial \Omega$. The set  $H^{1}_{*}(\Omega)$ is given by 
\[
 H^{1}_{*}(\Omega)= \left\{ u \in H^{1}(\Omega) \  \bigg| \  \int\limits_{\Omega} u~\mathrm{d}x\mathrm{d}y =0 \right\}.
\]
Thus we may rewrite the functional in~\eqref{functional} and consider
\begin{align*}
\F_\eps^{\lambda}(u) = \frac{1}{c_{0}}\int\limits_{\Omega} \frac{\epsilon}{2}|\nabla u|^{2} + \frac{1}{\epsilon} W(u)~\mathrm{d}x\mathrm{d}y  + \frac{\lambda}{2} \|u-\bar{m}\|_{H^{-1}(\Omega)}^{2}.
\end{align*}
 Doing so, we can formulate the gradient flow of the functional in~\eqref{functional} as
\begin{align}\label{gradientflow}
\left<\partial_{t}u,\phi\right>_{H^{-1}(\Omega)}=- \delta_{u;\phi}\F_\eps^{\lambda}(u) \ \forall \ \phi \in H_{*}^{1}(\Omega)
\end{align} 
with the first variation $ \delta_{u;\phi}\F_\eps^{\lambda}(u)$ in $u$ in the direction of $\phi$. This yields
\begin{align*}
\left<\partial_{t}u,\phi\right>_{H^{-1}(\Omega)} &= \int\limits_{\Omega} \left\{ \frac{\epsilon}{c_{0}} \Delta u - \frac{1}{c_{0}\epsilon} W'(u)-\lambda (-\Delta_{N})^{-1}(u-\bar{m})\right\}\phi\,\mathrm{d}x\mathrm{d}y
\end{align*}
for all test functions $\phi \in H^{1}_{*}(\Omega)$. 

Setting now $w= \Delta_{N}^{-1}\partial_{t}u+ \lambda \Delta_{N}^{-1}(u-\bar{m})$ and using the mass conservation of solutions of the following system, see, e.g., \cite{parsons2012numerical}, we are finally required to solve 
\begin{align}\label{system}
\int\limits_{\Omega}\left\{ \partial_{t}u-\Delta w + \lambda (u-\bar{m})\right\} \phi\, \mathrm{d}x\mathrm{d}y &= 0 , \notag \\
\int\limits_{\Omega} \left\{ w+ \frac{\epsilon}{c_{0}} \Delta u -\frac{1}{c_{0}\epsilon} W'(u)\right\}\phi\,\mathrm{d}x\mathrm{d}y &=0
\end{align}
for all test functions $\phi \in H^{1}(\Omega)$.

\subsection{Phase Field Connectedness}
It remains to include the possibility of using connected and simply connected perimeters in the numerical treatment of the Ohta-Kawasaki energy. Our method to enforce such a connectedness constraint for diffuse interfaces is based on the functional $\mathcal{C}_{\epsilon}$, introduced in~\cite{MR3590663, MR4011685}, and given by
\begin{align}\label{pathPenfunc}
\mathcal{C}_{\epsilon}^{(1)}(u) = \int\limits_{\Omega}\int\limits_{\Omega} \beta_{\epsilon}(u(x))\beta_{\epsilon}(u(y)) \operatorname{d}^{\psi_\eps(u)}(x,y) ~\mathrm{d}x\mathrm{d}y,
\end{align} 
where $\beta_{\epsilon}, \psi_{\epsilon}$ are continuous functions such that 
\begin{align*}
\beta_{\epsilon}, \psi_\eps \geq 0, \ \ \beta_{\epsilon}(z)= 0 \Leftrightarrow z \in [0,1-\alpha_\eps], \ \ \psi_\eps(z) > 0 \Leftrightarrow z \in [0,1-\alpha_\eps] 
\end{align*}
with $\alpha_\eps = \eps^s$ for some $0<s<1/2$ and $\operatorname{d}^\psi$ is a geodesic distance with local weight $\psi$.

Adding the functional in~\eqref{pathPenfunc} to a given functional in a diffuse interface approach then ensures approximate connectedness of the phase $\{u\approx 1\}$. Connectedness of the phase $\{1-u\approx 1\}=\{u\approx 0\}$ on the other hand can be achieved by adding the functional
\begin{align*}
\mathcal{C}_{\epsilon}^{(2)}(u) = \int\limits_{\Omega}\int\limits_{\Omega} \beta_{\epsilon}(1-u(x))\beta_{\epsilon}(1-u(y)) \,\operatorname{d}^{\psi_\eps(1-u)}(x,y) ~\mathrm{d}x\mathrm{d}y,
\end{align*}
thus adding both $\mathcal{C}_{\epsilon}^{(1)}(u)$ and $\mathcal{C}_{\epsilon}^{(2)}(u)$ to a given functional serves to keep the phase $\{u\approx 1\}$ simply connected in our two-dimensional setting.

Incorporating the connectedness constraint in the gradient flow dynamics in~\eqref{gradientflow} then leads to
\begin{align}\label{functional2}
\left<\partial_{t}u,\phi\right>_{H^{-1}(\Omega)}=- \left[\delta_{u;\phi}\F_\eps^{\lambda}(u)+\frac{\zeta_{1}}{\epsilon^{\kappa}}\delta_{u;\phi}\mathcal{C}_{\epsilon}^{(1)}(u) +\frac{\zeta_{2}}{\epsilon^{\kappa}}\delta_{u;\phi}\mathcal{C}_{\epsilon}^{(2)}(u)\right] \ \ \forall \ \phi \in H^{1}(\Omega)
\end{align}
for parameters $\kappa>0$, $\zeta_{1},\zeta_{2} \geq 0$. Considering now the system of equations in~\eqref{system} we are finally faced with solving 
\begin{align}\label{system1}
\int\limits_{\Omega}\left\{ \partial_{t}u-\Delta w + \lambda (u-\bar{m})\right\} \phi ~\mathrm{d}x\mathrm{d}y &= 0 , \notag \\
\int\limits_{\Omega} \left\{ w+ \frac{\epsilon}{c_{0}} \Delta u -\frac{1}{c_{0}\epsilon} W'(u)\right\}\phi ~\mathrm{d}x\mathrm{d}y - \left(\frac{\zeta_{1}}{\epsilon^{\kappa}}\delta_{u;\phi}\mathcal{C}_{\epsilon}^{(1)}(u)+\frac{\zeta_{2}}{\epsilon^{\kappa}}\delta_{u;\phi}\mathcal{C}_{\epsilon}^{(2)}(u)\right) &=0
\end{align}
for all test functions $\phi \in H^{1}(\Omega)$.

\subsection{Discretization}

To compute approximate solutions of the system in~\eqref{system1} numerically, we use P1-finite elements and obtain the system of equations
\begin{align*}
\left(\frac{u_{h}^{n}-u_{h}^{n-1}}{\Delta t},v_{h}\right) &= - (\nabla w_{h}^{n},\nabla v_{h})-  \lambda (u_{h}^{n}-\bar{m},v_{h}) \\
(w_{h}^{n},v_{h}) &= \frac{\epsilon}{c_{0}} (\nabla u_{h}^{n},\nabla v_{h}) + \frac{1}{c_{0}\epsilon}\big(W'(u_{h}^{n},u^{n-1}_{h}),v_{h}\big) \\
 & \quad \quad \quad + \zeta_{1} \epsilon^{-\kappa}\delta_{u^{n-1}_{h};v_{h}}\mathcal{C}_{\epsilon}^{(1)}(u^{n-1}_{h})\\
 & \quad \quad \quad  +\zeta_{2} \epsilon^{-\kappa}\delta_{u^{n-1}_{h};v_{h}}\mathcal{C}_{\epsilon}^{(2)}(u^{n-1}_{h}), 
\end{align*}
which uses a linearized version of the double well potential $W(u)$. Specifically, we use the approximation
\[
W'(u^{n})\approx \frac{u^{n}}{2}\left((u^{n-1}-1)^{2}+u^{n-1}(u^{n-1}-1)\right)= W'(u^{n},u^{n-1}). 
\]

This linearization was used in a similar way in~\cite{parsons2012numerical} for the numerical approximation of local minimizers of the Ohta-Kawasaki Energy. There, further relevant issues regarding stability and boundedness for a similar linearization are treated. 

The explicit treatment and discretization of the variation of the functionals $\mathcal{C}_{\epsilon}^{(1)}$ and $\mathcal{C}_{\epsilon}^{(2)}$ are discussed in detail in \cite{Dondl_18f}. As described there, we use a Dijkstra-type algorithm, based on ideas in \cite{MR2684290,MR3337998}, to compute the variation of a discretized geodesic distance.

\subsection{Numerical Results}\label{section numerical}

We now consider a fully discrete gradient flow of the functional in~\eqref{functional2}. For the numerical experiments, the functions $\beta_{\epsilon}$ and $\psi_\eps$ are given as in~\cite{MR4011685} 
\begin{align*}
\beta_\eps(s) &= \begin{cases}
0 & s\le 1-\alpha \\
\frac{c_1}{2}(s-1+\alpha)^2  & s>1-\alpha
\end{cases} \\
\quad\text{and}\\
\psi_\eps(s) &= \begin{cases} 
\frac{1}{2}(s-1+\alpha)^2  & s<1-\alpha\\
0 & s \le 1-\alpha,
\end{cases}
\end{align*}
respectively. The parameter $c_{1}$ is chosen such that $\int_{\alpha}^{1} \beta_{\epsilon}(s)~\mathrm{d}s =1$ and $\kappa$ in~\eqref{functional2} is set to $\kappa =2$. The value of $\alpha = \alpha_\eps$ changes slightly between experiments. 
  By incorporating Neumann boundary conditions as described above the mass of the initial condition is maintained during the evolution of the gradient flow. 

\subsubsection{Experiment 1} In the first numerical experiment we choose an initial condition which is approximately given by the characteristic function of the set $\{r< 0.02+0.45\cos(2\,\theta)\}$ with $r= \sqrt{x^{2}+y^{2}} $ and $\theta = \arctan(x,y)$, see Figure~\ref{fig1}. We set $\epsilon= 8 \cdot 10^{-3}$, $\tau=9.5 \cdot 10^{-9}$ and $\lambda=10606$. The mean value $\bar{m}$ is given by the initial condition as $\bar{m}\approx 0.178$. The parameters $\alpha$ and $\zeta_{1}$ are set to $\alpha=0.35$ and $\zeta_{1}=3.0$. The parameter $\zeta_{2}$ is set to zero  so we just ensure the phase $\{u\approx 1\}$ to be connected. The discretization is made up of approximately $4.6 \cdot 10^{4}$ P1 triangle elements on the square $\Omega = \big(-\frac{1}{2},\frac{1}{2}\big)$.

Without using a path-connectedness constraint two discs which repel each other form, see Figure~\ref{fig1}. They remain at a finite distance due to boundary effects. This represents a classical ``dynamically metastable'' solution to the minimization problem in~\eqref{functional2} without disconnectedness penalty, see, e.g., \cite{MR2854591}. Incorporating path-connectedness in the functional in~\eqref{functional2}, these balls cling to the Steiner-tree forming a dumbbell-like structure. 
\begin{figure}[ht!]\begin{center}
\raisebox{-0.5\height}{\includegraphics[height=3.25cm]{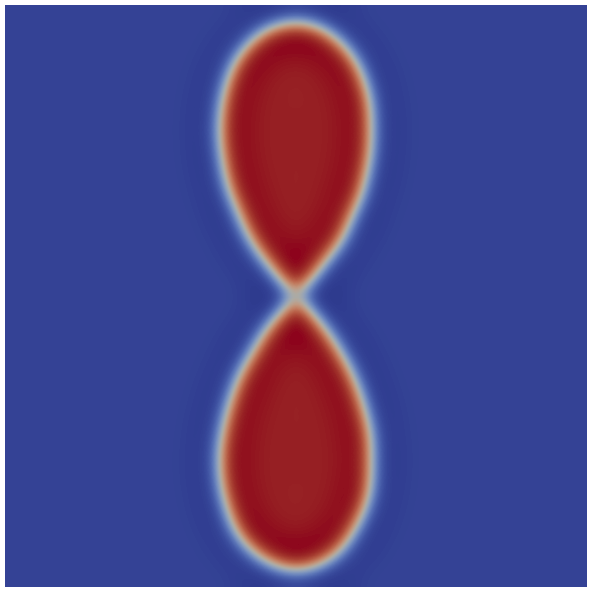}}
\raisebox{-0.5\height}{\includegraphics[height=3.25cm]{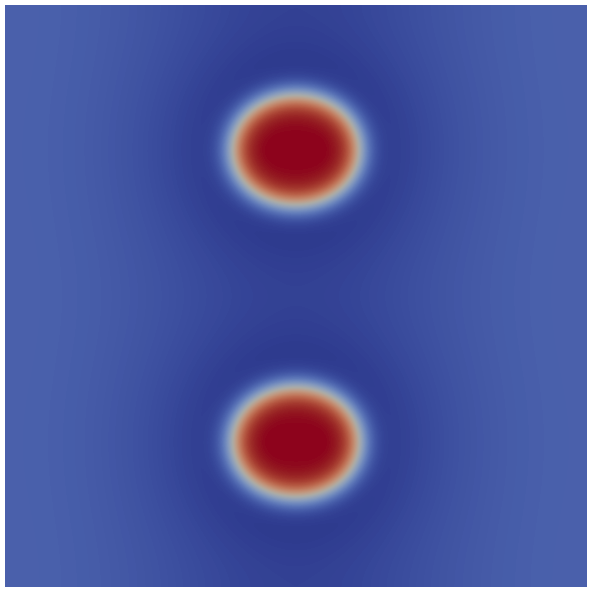}}
\raisebox{-0.5\height}{\includegraphics[height=3.25cm]{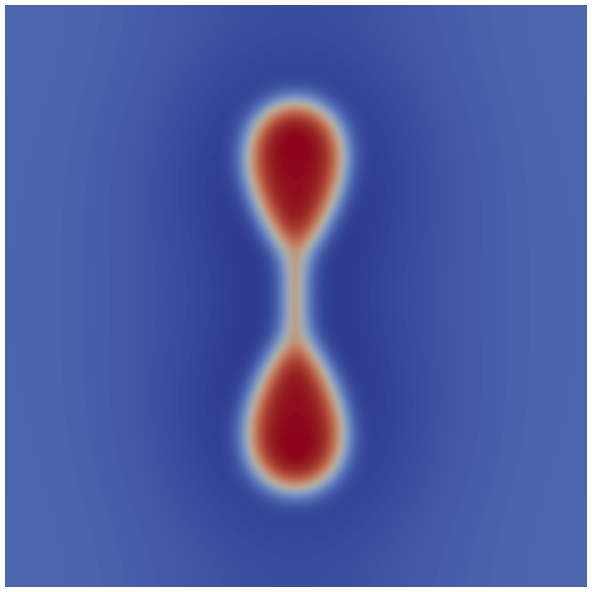}}
\raisebox{-0.5\height}{\includegraphics[height=2cm]{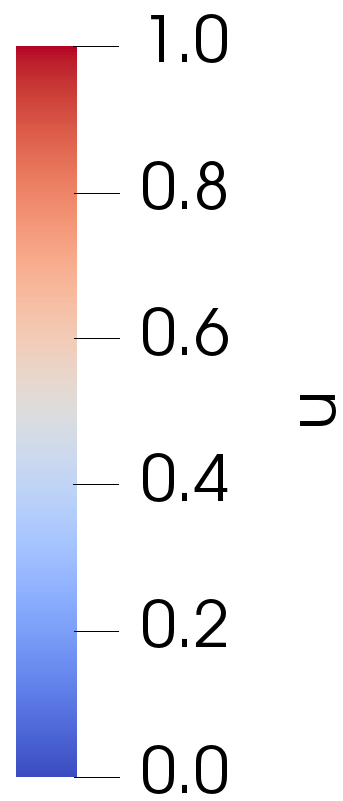}}
\caption{Results for a numerical example. From left to right:  Initial condition for $u$, ``dynamically metastable" state $u$ without disconnectedness penalty, stationary state with disconnectedness penalty. We use $\epsilon = 8 \cdot 10^{-3}$, $\lambda=10606$.}
\label{fig1}
\end{center}
\end{figure}

\begin{figure}[H]
\centering
  \qquad
    \subfigure[Without disconnectedness penalty.] {\includegraphics[width=0.46\textwidth]{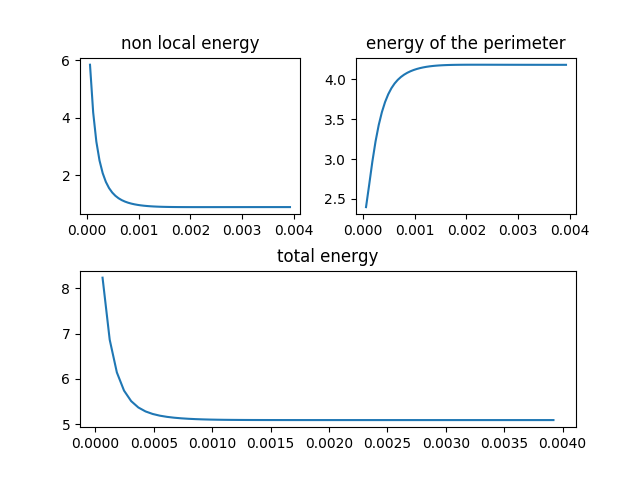}}\hspace{0.05em}
  \subfigure[With disconnectedness penalty.] {\includegraphics[width=0.46\textwidth]{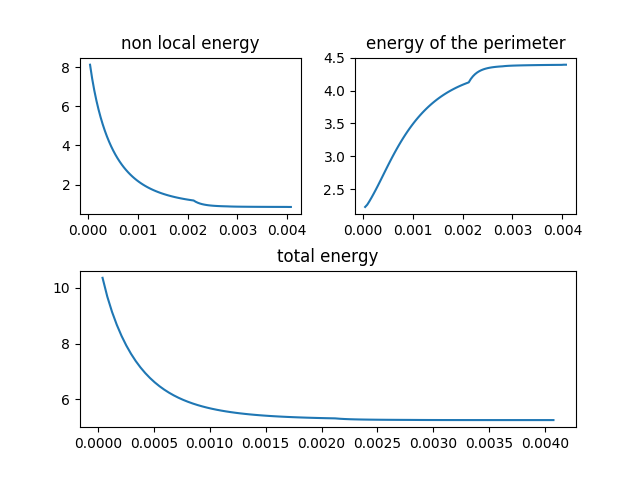}}
  \captionsetup{width=0.9\linewidth}
 \caption[]{Time evolution of the components of the free energy functional for experiment 1.}
\label{graph1}
\end{figure}

\subsubsection{Experiment 2} In the second experiment we set $\epsilon = 4\cdot 10^{-3}$, $\tau = 4.7 \cdot 10^{-9}$ and $\lambda =14849$. Further we decrease the mass $m$ and thus the initial condition is approximately given by the characteristic function of the set $\{r< 0.01+0.35\cos(2\,\theta)\}$ with the same notation as before, see Figure~\ref{fig2}. The mean value $\bar{m}$ is given by the initial condition as $\bar{m}\approx 0.1$. The parameters $\alpha$ and $\zeta_{1}$ are set to $\alpha=0.35$ and $\zeta_{1}=1.0$. The parameter $\zeta_{2}$ is set to zero so we just ensure the phase $\{u\approx 1\}$ to be connected. The discretization is again made up of approximately $4.6 \cdot 10^{4}$ P1 triangle elements on the square $\Omega = [-\frac{1}{2},\frac{1}{2}]$. 

The results of the second experiment are similar to the results from the first experiment, whereas the commingling of the two monomers is less pronounced and sharper interfaces are developed.
\begin{figure}[ht!]\begin{center}
\raisebox{-0.5\height}{\includegraphics[height=3.25cm]{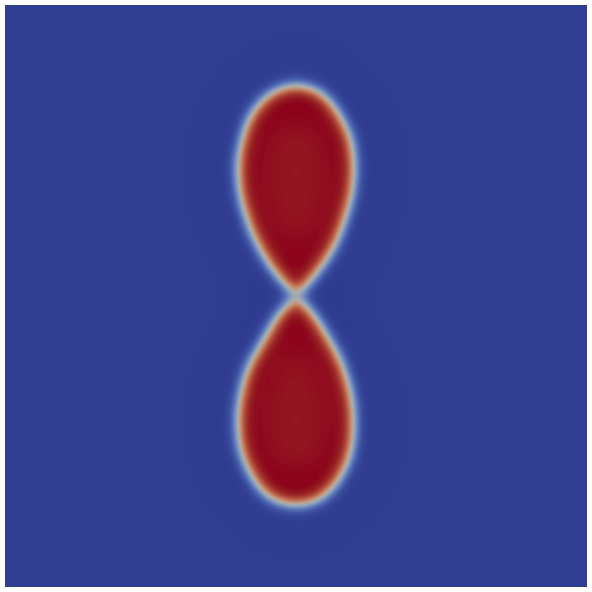}}
\raisebox{-0.5\height}{\includegraphics[height=3.25cm]{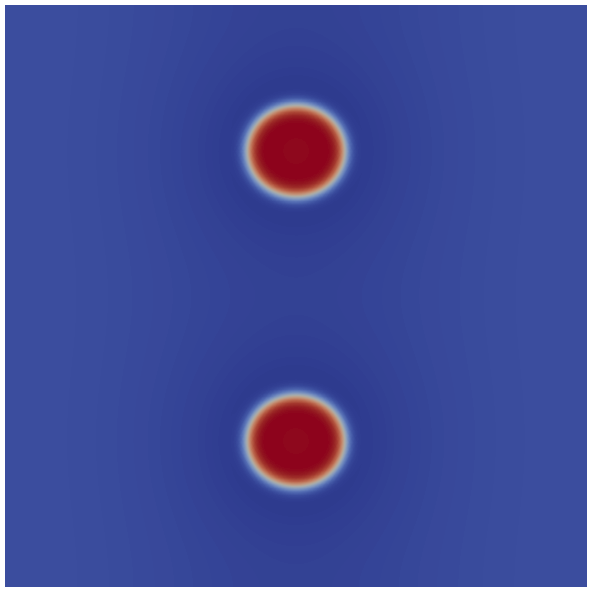}}
\raisebox{-0.5\height}{\includegraphics[height=3.25cm]{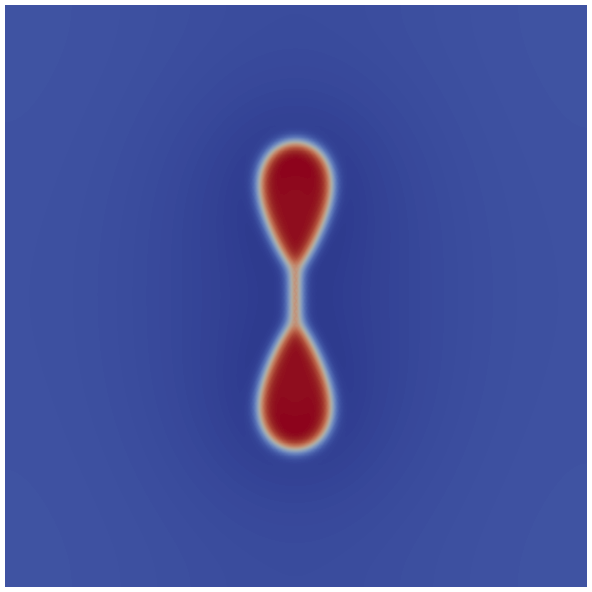}}
\raisebox{-0.5\height}{\includegraphics[height=2cm]{colorbar_u.png}}
\caption{Results for a numerical example. From left to right: Initial condition for $u$, ``dynamically metastable" state $u$ without disconnectedness penalty, stationary state with disconnectedness penalty. We use $\epsilon = 4 \cdot 10^{-3}$, $\lambda=14849$.}
\label{fig2}
\end{center}
\end{figure}

\begin{figure}[H]
\centering
  \qquad
    \subfigure[Without disconnectedness penalty.] {\includegraphics[width=0.46\textwidth]{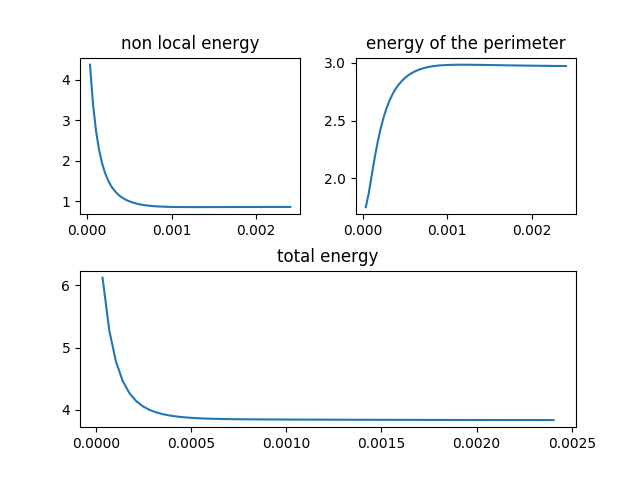}}\hspace{0.05em}
  \subfigure[With disconnectedness penalty.] {\includegraphics[width=0.46\textwidth]{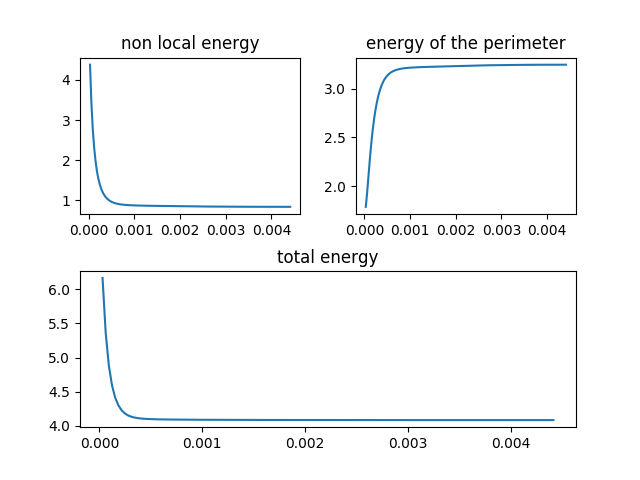}}
  \captionsetup{width=0.9\linewidth}
 \caption[]{Time evolution of the components of the free energy functional for experiment 2.}
\label{graph2}
\end{figure}

\begin{remark}
It should be mentioned that the steady states in the experiments with disconnectedness penalty indeed seem to be local minimizers of the energy in ~\eqref{functional2}. This effect can be explained by the observation that, in contrast to the case without disconnectedness penalty, every increase in the distance of the two droplets would also increase the perimeter of the phase $\{u\approx 1\}$. This is different to the ``dynamically metastable'' state occurring when abandoning the disconnectedness penalty, as every increase in the distance of the two droplets would decrease the whole energy in~\eqref{functional2}. 
\end{remark}

\subsubsection{Experiment 3} In a final third experiment we illustrated the effect when penalizing simple-connectedness of the phase $\{u\approx 1\}$. This is achieved by setting $\zeta_{1}=\zeta_{2}=0.01$ in~\eqref{functional2}. We now take the initial condition for $u$ as an approximation of the characteristic function of the set $\{r< 0.4+0.2\cos(2\,\theta)\}$, see Figure~\ref{fig3}. We further set $\epsilon = 3 \cdot 10^{-3}$, $\tau = 3 \cdot 10^{-9}$ and $\lambda =20000$. The parameter $\alpha$ is now set to $\alpha= 0.068$ and the discretization is made up of approximately $1.4 \cdot 10^{5}$ P1 triangle elements on the square $\Omega = [-1,1]$. The mean value $\bar{m}$ is thus given by the initial condition as $\bar{u}\approx 0.141$.

Due to the chosen initial condition and the value of $\lambda$, tubular structures occur in the simulation without disconnectedness penalty. These multiply-connected tubular structures are further typical ``dynamically metastable'' states of the Ohta-Kawasaki functional, see, e.g., \cite{MR2854591}. When we incorporate the simple-connectedness penalty, the tube tears open forming a simply connected structure for $\{u\approx 1\}$, see Figure~\ref{fig3}. As the impact of the boundedness of the reference domain is more visible than in the experiments before, the relevance of this experiment is more justified by effects inherited by the phase field model. An appropriate connection to the results in Section~\ref{section large mass} may not exist.

\begin{figure}[ht!]\begin{center}
\raisebox{-0.5\height}{\includegraphics[width = 3.25cm]{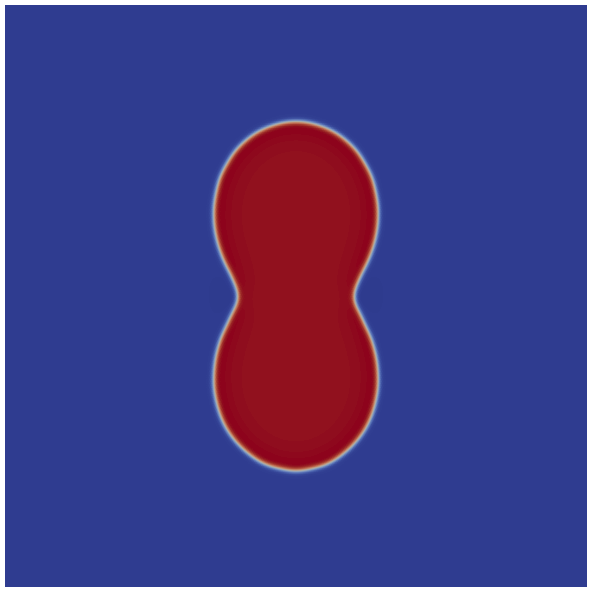}}
\raisebox{-0.5\height}{\includegraphics[width = 3.25cm]{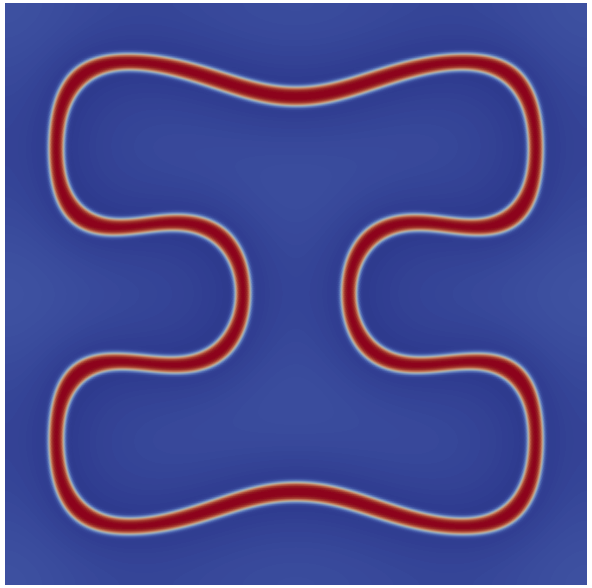}}
\raisebox{-0.5\height}{\includegraphics[width = 3.25cm]{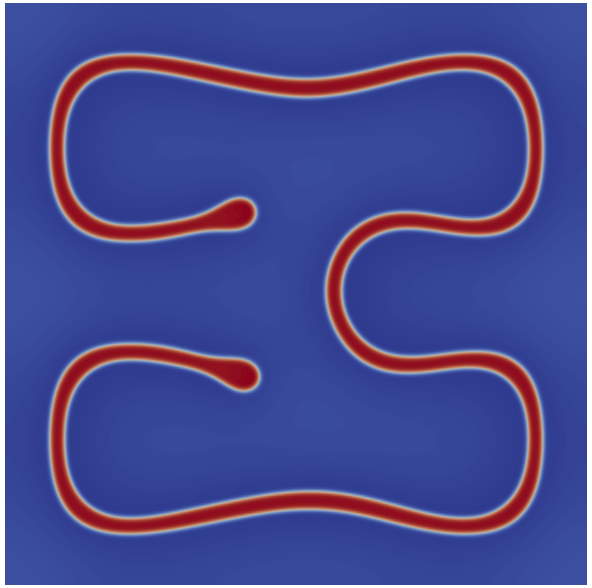}}
\raisebox{-0.5\height}{\includegraphics[height=2cm]{colorbar_u.png}}
\caption{Results for a numerical example. From left to right: Initial condition for $u$, ``dynamically metastable" state $u$ without disconnectedness penalty, ``dynamically metastable" state using simple-connectedness constraint. We use $\epsilon = 3 \cdot 10^{-8}$, $\lambda=20000$.}
\label{fig3}
\end{center}
\end{figure}

\begin{figure}[H]
\centering
  \qquad
    \subfigure[Without disconnectedness penalty.] {\includegraphics[width=0.46\textwidth]{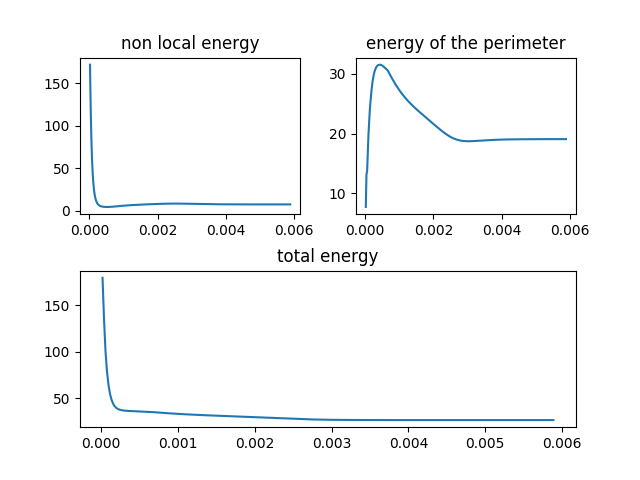}}\hspace{0.05em}
  \subfigure[With disconnectedness penalty.] {\includegraphics[width=0.46\textwidth]{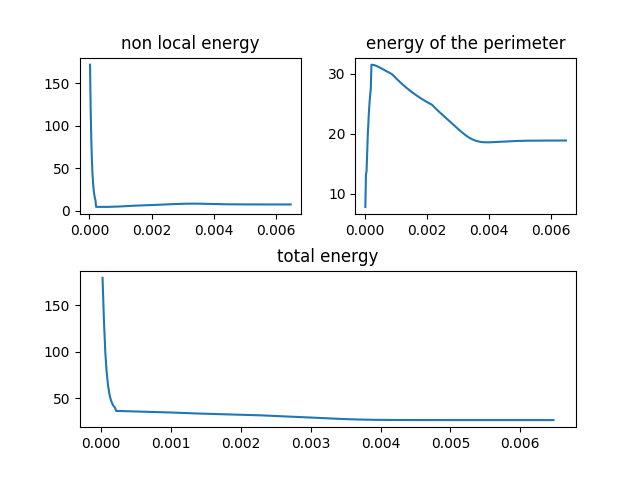}}
  \captionsetup{width=0.9\linewidth}
 \caption[]{Time evolution of the components of the free energy functional for experiment 3.}
\label{graph3}
\end{figure}

\bibliographystyle{alphaabbr}
\bibliography{quellen}

\end{document}